\documentclass[a4paper, 11pt]{article}
\usepackage[mathscr]{eucal}
\usepackage{amssymb}
\usepackage{latexsym}
\usepackage{amsthm}
\usepackage{amsmath}
\usepackage[dvips]{graphicx}
\usepackage{psfrag}
\usepackage{a4wide}

\newtheorem{theorem}{Theorem}[section]
\newtheorem{proposition}[theorem]{Proposition}
\newtheorem{lemma}[theorem]{Lemma}
\newtheorem{corollary}[theorem]{Corollary}
\theoremstyle{definition}

\newtheorem{problem}[theorem]{Problem}

\makeatletter
\@addtoreset{equation}{section}

\makeatother

\title{Maximal $2$-distance sets containing the regular simplex}  

\author{
Hiroshi Nozaki and Masashi Shinohara
}
\begin{document}
\maketitle

\renewcommand{\thefootnote}{\fnsymbol{footnote}}
\footnote[0]{
\\ \noindent
{\it Hiroshi Nozaki}: 
	Department of Mathematics Education, 
	Aichi University of Education, 
	1 Hirosawa, Igaya-cho, 
	Kariya, Aichi 448-8542, 
	Japan.
	hnozaki@auecc.aichi-edu.ac.jp.

\noindent
{\it Masashi Shinohara}: 
	Department of Education, 
	Faculty of Education, 
	Shiga University,
	2-5-1 Hiratsu, Otsu, Shiga 520-0862, 
	Japan.
	shino@edu.shiga-u.ac.jp
}

\begin{abstract}
A finite subset $X$ of the Euclidean space is called an $m$-distance set if the number of distances between two distinct points in $X$ is equal to $m$. 
An $m$-distance set $X$ is said to be maximal if 
any vector cannot be added to $X$ while maintaining the $m$-distance condition. We investigate a necessary and sufficient condition for vectors to be added to a regular simplex such that the set has only $2$ distances. 
We construct several $d$-dimensional maximal $2$-distance sets that contain a $d$-dimensional regular simplex. 
In particular, there exist infinitely many maximal non-spherical $2$-distance sets that contain both the regular simplex and the representation of a strongly resolvable design.  
The maximal $2$-distance set has size $2s^2(s+1)$, and the dimension is $d=(s-1)(s+1)^2-1$, where $s$ is a prime power.   
\end{abstract}
\section{Introduction}
A finite subset $X$ of the Euclidean space $\mathbb{R}^d$ is called an {\it $m$-distance set} if $|A(X)|=m$, where 
\[A(X)=\{\mathfrak{d}(\mathbf{x},\mathbf{y})\colon\, \mathbf{x},\mathbf{y} \in X, \mathbf{x} \ne \mathbf{y} \}
\]
and $\mathfrak{d}(\mathbf{x},\mathbf{y})$ is the Euclidean distance of $\mathbf{x}$ and $\mathbf{y}$. 
The size of an $m$-distance set in $\mathbb{R}^d$ is bounded above by $\binom{d+m}{m}$ \cite{BBS,B84}. 
The major problem of $m$-distance sets is to determine the largest possible $m$-distance set for given dimension $d$. 
Let $N=\binom{d+m-1}{m-1}+\binom{d+m-2}{m-2}$.
If the size of an $m$-distance set in $\mathbb{R}^d$ is at least $2N$, then 
\[
s_j=\prod_{i=1,\ldots,m,i\ne j}\frac{\alpha_j^2}{\alpha_j^2-\alpha_i^2}
\]
is an integer, where $A(X)=\{\alpha_1,\ldots, \alpha_m\}$ \cite{N11}. 
This result for $m=2$ is proved by Larman--Rogers--Seidel \cite{LRS77}.
The value $s_j$ is called the {\it LRS ratio}. 
The absolute value of the LRS ratio is not large for given dimension $d$, namely $|s_j| \leq 1/2+\sqrt{N^2/(2N-2)+1/4}$ \cite{N11}. Moreover, for given integers $s_j$ ($1 \leq j \leq m$),  we can uniquely determine the distances $\alpha_i$ ($1\leq i \leq m$) up to scale \cite{N11}. 
In particular, for $m=2$, if we fix the integer $s_2=\alpha_2^2/(\alpha_2^2-\alpha_1^2)$, then we can determine $\alpha_2$ for $\alpha_1=1$. 
The LRS ratio is one of useful parameters to characterize large $m$-distance sets. 
 There are only finitely many $m$-distance sets whose size is at least $2N$ \cite{N11}. 
Largest $m$-distance sets are known for 
$(m,d)=(1,\text{any}), (\leq 6,2), (2,\leq 8),(3,3),(4,3),(3,4)$ \cite{ES66,EF96,L97,S04,S08,Spre,SO20,W12}.

There are large $m$-distance sets on the sphere $S^{d-1}$ that are obtained from representations of association schemes \cite{BB05, BIb,DGS77}. 
In particular, the representations of the Johnson schemes of class 2 on $d+1$ points are largest $2$-distance sets in
$S^{d-1}$ of size $d(d+1)/2$ except for $d=(2t+1)^2-3$ and $t \in \mathbb{N}$ \cite{GW18}. 
A systematic construction of non-spherical $m$-distance sets is not known in the literature. 
One of ad hoc constructions is to add vectors to a non-maximal spherical $m$-distance set while maintaining the $m$-distance condition.  Here an $m$-distance set $X$ in $\mathbb{R}^d$ is ${\it maximal}$ if there does not exist $\mathbf{x} \in \mathbb{R}^d \setminus X$ such that $X \cup \{\mathbf{x}\}$ is still $m$-distance.  
Maximal $m$-distance sets for $m=2,3,4$ containing the representations of the Johnson schemes of class $m$ and the Hamming schemes of class $m$ are investigated in \cite{BSS12} and \cite{AHNY17}, respectively. 
In this paper, we consider $2$-distance sets in $\mathbb{R}^d$ that contain a $d$-dimensional regular simplex. A largest $2$-distance set in $\mathbb{R}^8$ with 45 points \cite{L97} attains the bound $|X|\leq \binom{d+m}{m}$, and it contains a $8$-dimensional regular simplex.   

Let $R_d$ be a $d$-dimensional regular simplex with $d+1$ points. 
There are many choices of $\mathbf{x} \in \mathbb{R}^d$ such that $R_d \cup \{\mathbf{x } \}$ has only 2 distances.  
We suppose that the LRS ratio is an integer in order to find large 2-distance sets. 
First we prove that if a given LRS ratio $s$ is an integer, then    
there exist only finitely many dimensions $d$
where there exists a $2$-distance set containing $R_d$.   
For $s=2$, the possible dimensions are $d=7,8$.  We classify largest $2$-distance sets containing $R_d$ for $d=7,8$. For $s= 3$, the possible dimensions are $d=23,24,26,31,48$. 
In this case, we construct maximal $2$-distance sets containing $R_d$ by adding the representation of a quasi-symmetric design, see Table~\ref{tb:1}.  
In particular, there exist infinitely many dimensions $d$ where there exists a maximal $2$-distance set  containing both $R_d$ and the representation of a strongly resolvable design. 
\begin{center}
\begin{table} \label{tb:1}
\caption{Maximal 2-distance set that contains $R_d$}
\begin{tabular}{ccccccc}\label{tab:1}
$d$&size&  LRS ratio     & added set   \\ \hline 
7&29 &2& $J(8,3)$ (Largest 2-distance set in $\mathbb{R}^7$) \\
8&24 &2& Hadamard matrix of order $8$\\
8&30&2&   Largest 2-intersecting family\\
8&45&2& $J(9,2)$ (Largest 2-distance set in $\mathbb{R}^8$) \\
23 &144 &   3 &   2-$(21,7,12)$ design   \\ 
24 &278& 3  & 4-$(23,7,1)$ design  \\
26 & 280&3 & 4-$(23,7,1)$ design  \\
26 &280&3  &The complement of 4-$(23,7,1)$ design  \\
31 &110&3  &  3-$(22,6,1)$  design  \\
31 &286&3   &  The complement of 4-$(23,7,1)$ design  \\ 
48 &302&3   & The complement of 4-$(23,7,1)$ design  \\  
$s^3+s^2-s-2$ & $2s^2(s+1)$ & $s$ &  2-$(s^3,s^2,s+1)$ design
\end{tabular}
\end{table}  
\end{center}

\section{Vectors that can be added to the regular simplex $R_d$} 
In this section, 
we determine the form of a vector in $\mathbb{R}^d$ that can be added to a fixed $d$-dimensional regular simplex (the coordinates are fixed later) so that the number of distances are 2. 
Let $\{\mathbf{e}_1,\ldots, \mathbf{e}_{d+1}\}$ be the standard basis of the Euclidean space $\mathbb{R}^{d+1}$. 
Let $H_d$ denote the affine space  $\{(x_1,\ldots,x_{d+1})\in \mathbb{R}^{d+1} \mid \sum_{i=1}^{d+1} x_i =1\}$. Note that $H_d$ is isometric to $\mathbb{R}^d$. 
Let $R_d$ denote the set $\{\mathbf{e}_1,\ldots, \mathbf{e}_{d+1}\}$, which is interpreted as a $d$-dimensional regular simplex in $H_d$.  The set $R_d$ is a 1-distance set with distance $\sqrt{2}$. We would like to consider a maximal $2$-distance set in $H_d$ that contains $R_d$. 

Suppose there exists $\mathbf{x}=(x_1,\ldots, x_{d+1}) \in H_d$ such that $R_d \cup \{ \mathbf{x} \}$ has  only 2 distances $\sqrt{2}$, $\sqrt{\alpha}$. 
For distinct $i,j \in \{1,\ldots, d+1\}$, it follows that  
\[
\mathfrak{d}(\mathbf{x},\mathbf{e}_i)^2-\mathfrak{d}(\mathbf{x},\mathbf{e}_j)^2=2(x_i-x_j) \in \{0, \pm (\alpha-2)\}, 
\]
and $|x_i-x_j|\in \{0,|(\alpha-2)/2|\}$.  
Let $\beta=(\alpha-2)/2$, and note that $\beta \ne -1,0$ from $\alpha\ne 0,2$. 
 The vector $\mathbf{x}$ must be an element of  the set
\[
T_d(k,\beta)=\{\mathbf{x} \in H_d \colon\, \forall i, x_i \in \{c,c+\beta \}, |N(\mathbf{x},c)|=k  \},   
\]
for $k \in \{1, \ldots, d+1\}$, where  $N(\mathbf{x},a)=\{i \colon\, x_i=a\}$, and 
\[
c=\frac{1}{d+1}-\frac{d+1-k}{d+1} \beta. 
\]
For $i\in N(\mathbf{x},c)$, $j \in N(\mathbf{x},c+\beta)$, we have
\[
\mathfrak{d}(\mathbf{e}_i,\mathbf{x})^2-\mathfrak{d}(\mathbf{e}_j, \mathbf{x})^2=2 \beta=\alpha-2.
\]
This implies $\mathfrak{d}(\mathbf{e}_i,\mathbf{x})^2=\alpha$ and $\mathfrak{d}(\mathbf{e}_j, \mathbf{x})^2=2$. It follows from $\mathfrak{d}(\mathbf{e}_j,\mathbf{x})^2=2$ that
\begin{equation*}
\beta=
\begin{cases}
\frac{k\pm \sqrt{k(d+1)(d+2-k)}}{k(d+1-k)}\ne 0,
\qquad \text{if $1\leq k \leq d$},\\
   -\frac{d+2}{2(d+1)}, \qquad \text{if $k=d+1$}.
\end{cases}
 \end{equation*}
From  $1\leq k \leq d+1$, we have 
\begin{equation*} 
\frac{k- \sqrt{k(d+1)(d+2-k)}}{k(d+1-k)}
=-\frac{d+2}{k+\sqrt{k(d+1)(d+2-k)}}
\geq -1,
\end{equation*}
and the equality holds only if $k=1$.  By $\beta \ne -1$, it follows that $\beta=1+2/d$ for $k=1$. 
Therefore we have the following. 
\begin{proposition} \label{prop:1}
Let $d$, $k$ be integers with $d\geq 2$, $1\leq k \leq d+1$.  
Let $\beta$ be defined to be  
\begin{equation} \label{eq:beta}
\beta=\begin{cases}
\frac{k\pm \sqrt{k(d+1)(d+2-k)}}{k(d+1-k)}, \qquad \text{if $2\leq k \leq d$}, \\
1+\frac{2}{d}, \qquad \text{if $k=1$}, \\
 -\frac{d+2}{2(d+1)}, \qquad \text{if $k=d+1$}. 
\end{cases}
\end{equation}
Then, $\mathbf{x}$ is an element of  $T_d(k,\beta)$
for some $k$ and $\beta$ satisfying the conditions above if and only if $R_d \cup \{\mathbf{x} \}$ is a $2$-distance set in $H_d$.  
\end{proposition}
We consider when 2 elements $\mathbf{x}$, $\mathbf{y}$ of $T_d(k,\beta)$ can be simultaneously added to $R_d$ while maintaining the $2$-distance condition. Let $\mathfrak{l}=\mathfrak{l}(\mathbf{x}, \mathbf{y})=|\{i \colon\, x_i\ne y_i\}|/2$ for $\mathbf{x}=(x_1,\ldots, x_{d+1})$, $\mathbf{y}=(y_1,\ldots, y_{d+1})$. 
One has
\[
\mathfrak{d}(\mathbf{x}, \mathbf{y})^2=2 \beta^2\mathfrak{l}  \in \{2,\alpha \},
\] 
and 
$
\mathfrak{l}\in \{1/\beta^2,(\beta+1)/\beta^2\}.
$
Thus the following follows. 
\begin{proposition} \label{prop:l}
Let $\mathfrak{l}$, $T_d(k,\beta)$, and $R_d$ be defined as above. 
Suppose $d$, $k$, and $\beta$ satisfy the condition from Proposition~\ref{prop:1}. 
Let $X$ be a subset of $T_d(k,\beta)$. 
Then, 
$R_d \cup X$ is a $2$-distance set 
if and only if 
$\mathfrak{l}(\mathbf{x}, \mathbf{y}) \in \{1/\beta^2,(\beta+1)/\beta^2\}$ for any $\mathbf{x}, \mathbf{y} \in X$ with $\mathbf{x} \ne \mathbf{y}$. 
\end{proposition}
From Proposition \ref{prop:l}, 
we would like to construct a large 2-distance subset of $T_d(k,\beta)$ with only particular distances $1/\beta^2$ and $(\beta+1)/\beta^2$ to obtain a large 2-distance subset of $H_d$  containing $R_d$. 

\label{sec:1}

\section{Maximal 2-distance sets that contain the regular simplex}
\label{sec:3}

From Section~\ref{sec:1}, 
a construction of large 2-distance subsets of $T_d(k,\beta)$ gives large 2-distance sets in $H_d$ that contain the $d$-dimensional regular simplex $R_d$. 
However, if dimension $d$ is not small,  it is a hard problem to construct large 
$m$-distance sets.  Indeed, $T_d(k,\beta)$ is isometric to the Johnson association scheme $J(d+1,k)$, and largest $m$-distance sets in $J(d+1,k)$ are investigated in \cite{BM11,MN11}. There is no systematic construction of $m$-distance sets in $J(d+1,k)$ except for representations of association schemes.  
We suppose the LRS ratio is an integer in order to find large 2-distance sets. 

Larman--Rogers--Seidel \cite{LRS77} proved that 
if a $2$-distance set $X$ in $\mathbb{R}^d$ has at least 
$2d+4$ points,  then for 2 distances $a,b$ with $a<b$,  
there exists an integer $s$ such that $a^2 /b^2= (s-1)/s$ 
and $2\leq s \leq 1/2+\sqrt{d/2}$. 
The condition $|X| \geq 2d+4$ was improved to $|X| \geq 2d+2$ \cite{N81}. 
The integer $s$ is called the {\it LRS ratio} of a $2$-distance set. 
We would like to construct a maximal $2$-distance set $X$   
with $|X| \geq 2d+2$ that contains 
the regular simplex.  
In the first part of this section, 
we prove that if a given LRS ratio $s$ is an integer, then there are only finitely many choices of $(d,k,\beta)$.   The following is the key theorem. 
\begin{theorem} \label{lem:key}
Let  $R_d$ be the regular simplex $\{\mathbf{e}_1,\ldots, \mathbf{e}_{d+1}\}$. Let $T_d(k,\beta)=\{\mathbf{x} \in H_d \colon\, \forall i, x_i \in \{c,c+\beta \}, |N(\mathbf{x},c)|=k  \}$, where  $N(\mathbf{x},a)=\{i \colon\, x_i=a\}$ and 
$
c=(1-(d+1-k)\beta)/(d+1)  
$.
Suppose $R_d \cup \{ \mathbf{x} \} $ has only two distances $\sqrt{2}, \sqrt{\alpha}$
for $\mathbf{x} \in T_d(k,\beta)$. Then the following follow. 
\begin{enumerate}
\item Suppose $\alpha<2$ and $s\geq 2$ satisfy $\alpha /2 =(s-1)/s$, namely the LRS ratio is $s$.  Then 
$2 \leq k \leq d$, $k\ne s^2$, $\beta=-1/s$, and 
\begin{equation}\label{eq:thm3.1_1}
d=k+s^2-2s-1 +\frac{s^2(s-1)^2}{k-s^2}. 
\end{equation}
\item Suppose $\alpha>2$ and $s\geq 2$ satisfy $ 2/\alpha =(s-1)/s$, namely the LRS ratio is $s$. Then $2 \leq k \leq d$, $d+2\geq k+s$, $k\ne (s-1)^2$, $\beta=1/(s-1)$, and 
\begin{equation} \label{eq:thm3.1_2}
d=k+s^2-2 +\frac{s^2(s-1)^2}{k-(s-1)^2}. 
\end{equation}
\end{enumerate}
\end{theorem}
\begin{proof}
Since $R_d \cup \{ \mathbf{x} \} $ is a $2$-distance set, $\beta$ can be expressed by $d$ and $k$ by Proposition~\ref{prop:1}. 

(1) Suppose $\alpha /2 =(s-1)/s$, namely $\beta =(\alpha-2)/2=-1/s<0$. 
For $k=1$, we have $\beta=1+2/d>0$, which contradicts $\beta<0$. 
For $k=d+1$, we have 
\[
\beta=-\frac{d+2}{2(d+1)}=-\frac{1}{s}, 
\]  
which is transformed to $(s-2)d=2-2s$. For $s\geq 2$, we have $(s-2)d\geq 0$ and $2-2s<0$, a contradiction. 
For $2 \leq k \leq d$, we have
\[
\beta=\frac{k-\sqrt{k(d+1)(d+2-k)}}{k(d+1-k)}=-\frac{1}{s}, 
\]
which is transformed to 
\begin{equation} \label{eq:3-1}
\big( (k-s^2)d-(k^2-(2s+1)k+2s^2) \big) \big(d-(k-1)\big)=0.
\end{equation}
Since $k\ne d+1$, we have 
\[
d=\frac{k^2-(2s+1)k +2s^2}{k-s^2}=k+s^2-2s-1+\frac{s^2(s-1)^2}{k-s^2}
\]
for $k\ne s^2$. 
For $k=s^2$, we have $k^2-(2s+1)k+2s^2=s^2(s-1)^2\ne 0$, which contradicts 
\eqref{eq:3-1}. 
 
(2) Suppose $2 /\alpha =(s-1)/s$, namely $\beta =(\alpha-2)/2=1/(s-1)>0$. 
For $k=1$, we have $\beta=1+2/d=1/(s-1)$. It is transformed to 
$s=1+d/(d+2)<2$, a contradiction.  
For $k=d+1$, we have 
\[
\beta=-\frac{d+2}{2(d+1)}<0, 
\]  
which contradicts $\beta>0$.  
For $2 \leq k \leq d$, we have
\[
\beta=\frac{k+\sqrt{k(d+1)(d+2-k)}}{k(d+1-k)}=\frac{1}{s-1}, 
\]
which is transformed to 
\begin{equation} \label{eq:3-2}
\big( (k-(s-1)^2)d-(k^2+(2s-3)k+2(s-1)^2) \big) \big(d-(k-1)\big)=0
\end{equation}
with $d+2\geq k+s$. 
Since $k\ne d+1$, we have 
\[
d=\frac{k^2+(2s-3)k +2(s-1)^2}{k-(s-1)^2}=k+s^2-2+\frac{s^2(s-1)^2}{k-(s-1)^2}
\]
for $k\ne (s-1)^2$. 
For $k=(s-1)^2$, we have $k^2+(2s-3)k+2(s-1)^2=s^2(s-1)^2\ne 0$, which contradicts 
\eqref{eq:3-2}. 
\end{proof}
\begin{corollary}\label{coro:key}
Let $T_d(k,\beta)$ and $R_d$ be defined as in Theorem~\ref{lem:key}.  Let $s$ be an integer at least 2, and $\mathbf{x} \in T_d(k,\beta)$.
Then there exist only finitely many choices of $(d,k,\beta)$ such
that $R_d \cup \{\mathbf{x} \}$ is a $2$-distance set with LRS ratio $s$.  
\end{corollary}
\begin{proof}
Since $R_d \cup \{ \mathbf{x} \} $ is a $2$-distance set, $\beta$ can be expressed by $d$ and $k$ by Proposition~\ref{prop:1}. 
By Theorem~\ref{lem:key}, 
$(d,k)$ satisfies \eqref{eq:thm3.1_1} or \eqref{eq:thm3.1_2}. 
Since $d$ is an integer, $s^2(s-1)^2$ can be divided by $k-s^2$ or $k-(s-1)^2$. This implies the assertion.    
\end{proof}
If the cardinality of a 2-distance set is at least $2d+2$, then the LRS ratio $s$ is an integer and there are only finitely 
many pairs $(d,k)$ by Corollary~\ref{coro:key}.  
The distances of 
a 2-distance subset of $T_d(k,\beta)$ added to $R_d$ can be expressed by LRS ratio $s$. 


\begin{lemma} \label{lem:l}
Let $T_d(k,\beta)$ and $R_d$ be defined as in Theorem~\ref{lem:key}. 
 Let $\mathbf{x}=(x_1,\ldots,x_{d+1}), \mathbf{y}= (y_1,\ldots, y_{d+1}) \in T_d(k, \beta)$ with $\mathbf{x}\ne \mathbf{y}$ and $\mathfrak{l}(\mathbf{x}, \mathbf{y})=|\{i \mid x_i \ne y_i \}|/2$. 
\begin{enumerate}
\item Suppose  $\alpha<2$. 
The set $R_d \cup \{\mathbf{x}, \mathbf{y}\} $ has only two distances $\sqrt{2}$, $\sqrt{\alpha}$ and the LRS ratio $s$ is an integer  
if and only if $(d,k)$ satisfies equation \eqref{eq:thm3.1_1}, $\beta=-1/s$, and $\mathfrak{l}(\mathbf{x},\mathbf{y}) \in \{s^2, s(s-1)\}$. 
\item Suppose  $\alpha>2$. 
The set $R_d \cup \{\mathbf{x}, \mathbf{y}\} $ has only two distances $\sqrt{2}$, $\sqrt{\alpha}$ and the LRS ratio $s$ is an integer  
if and only if $(d,k)$ satisfies equation \eqref{eq:thm3.1_2}, 
$\beta=1/(s-1)$ and $\mathfrak{l}(\mathbf{x},\mathbf{y}) \in \{(s-1)^2, s(s-1)\}$.   
\end{enumerate}
\end{lemma}
\begin{proof}
By Theorem~\ref{lem:key} 
and Proposition~\ref{prop:l}, the assertion follows. 
\end{proof}
Conditions \eqref{eq:thm3.1_1} and \eqref{eq:thm3.1_2}
in Theorem~\ref{lem:key} are quadratic in $k$ so there are at most two values of $k$ for fixed $d$ and $s$. 
In the following lemma, we explicitly connect these two values. 
\begin{lemma}\label{lem:two_k}
\begin{enumerate}
\item Let $d(k)=k+s^2-2s-1 +s^2(s-1)^2/(k-s^2)$ and $k'=s^2(s-1)^2/(k-s^2)+s^2$. 
Then $d(k')=d(k)$ holds. 
\item Let $d(k)=k+s^2-2 +s^2(s-1)^2/(k-(s-1)^2)$ and $k'=s^2(s-1)^2/(k-(s-1)^2)+(s-1)^2$. 
Then $d(k')=d(k)$ holds. 
\end{enumerate}
\end{lemma}
\begin{proof}
By direct calculation, we can prove $d(k')=d(k)$. 
\end{proof}
We show an equivalent condition for 2 vectors $\mathbf{x}\in T_d(k, \beta)$ and $\mathbf{y} \in T_d(k',\beta)$ to 
be simultaneously added to $R_d$. 
\begin{lemma}\label{lem:m}
Let $T_d(k,\beta)$ and $R_d$ be defined as in Theorem~\ref{lem:key}.  
Let $\mathbf{x}=(x_1,\ldots, x_{d+1})\in T_d(k,\beta)$, 
$\mathbf{y}=(y_1,\ldots, y_{d+1}) \in T_d(k',\beta)$ and 
$\mathfrak{m}(\mathbf{x},\mathbf{y})=|\{ i \colon\, x_i=c,y_i=c' \}|$, 
where $c=(1-(d+1-k)\beta)/(d+1)$ and $c'=(1-(d+1-k')\beta)/(d+1)$. 
\begin{enumerate}
\item Suppose $\alpha <2 $. The set $R_d\cup \{\mathbf{x},\mathbf{y}\}$ has only two distances $\sqrt{2}$, $\sqrt{\alpha}$ and the LRS ratio $s$ is an integer if and only if $(d,k)$ satisfies equation \eqref{eq:thm3.1_1}, $\beta=-1/s$, $k'=s^2(s-1)^2/(k-s^2)+s^2 \ne k$, and $\mathfrak{m}(\mathbf{x},\mathbf{y}) \in \{s^2,s(s-1)\}$.
\item  Suppose $\alpha >2$. The set $R_d\cup \{\mathbf{x},\mathbf{y}\}$ has only two distances $\sqrt{2}$, $\sqrt{\alpha}$ and the LRS ratio $s$ is an integer if and only if $(d,k)$ satisfies equation \eqref{eq:thm3.1_2}, 
$\beta=1/(s-1)$, $k'=s^2(s-1)^2/(k-(s-1)^2)+(s-1)^2\ne k$, and $\mathfrak{m}(\mathbf{x},\mathbf{y}) \in \{(s-1)^2,s(s-1)\}$.
\end{enumerate}
\end{lemma}
\begin{proof}
(1) By direct calculations, the square of the distance between  $\mathbf{x}\in T_d(k,\beta)$ and 
$\mathbf{y} \in T_d(k',\beta)$ is equal to $-2\mathfrak{m}(\mathbf{x},\mathbf{y})/s^2-2/s+4$. 
Since it should be $2$ or $\alpha$, the assertion follows. 

(2) By direct calculations, the square of the distance between  $\mathbf{x}\in T_d(k,\beta)$ and 
$\mathbf{y} \in T_d(k',\beta)$ is equal to $-2\mathfrak{m}(\mathbf{x},\mathbf{y})/(s-1)^2-2/(s-1)+4$. 
Since it should be $2$ or $\alpha$, the assertion follows. 
\end{proof}

We give several maximal 2-distance sets $X$ containing $R_d$ with $|X|\geq 2d+2$ in the following subsections. 
From Lemmas~\ref{lem:l} and \ref{lem:m}, 
we know the exact values of distances of a subset of $T_d(k,\beta) \cup T_d(k',\beta)$ that can be added to $R_d$. However it is hard to construct such a large subset especially for large $d$.  
It is so far hopeless to classify maximal 2-distance sets that contain the regular simplex. 
For $s=2$, since the dimension $d$ is not large,  we can determine the largeset $2$-distance set containing $R_d$ for each possible $d$. 
For $s\geq 3$, we construct maximal $2$-distance sets containing $R_d$ by adding the representation of  a quasi-symmetric design. 

\subsection{LRS ratio $s=2$, distances $1,\sqrt{2}$ ($\alpha=1$)}
By Theorem~\ref{lem:key}, 
$
d=k-1+4/(k-4)
$. 
Since $d$, $k$ are natural numbers, the possible $(d,k)$ are $(7,6)$, $(8,5)$, or $(8,8)$.  
By Lemma~\ref{lem:two_k}, there are at most 2 choices of $k$ for given $d$, in this case $k=5$ and $k'=8$ for $d=8$, and $k=k'=6$ for $d=7$. 
The set $T_d(k,\beta)$ is identified with the Johnson scheme $J(d+1,k)$, which is the set of 
all $k$-subsets of a $(d+1)$-point set equipped with the distance $\mathfrak{l}(\mathbf{x},\mathbf{y})$. Here $J(d+1,k)$ can be identified with a set of $(0,1)$-vectors in $\mathbb{R}^{d+1}$ with $k$ ones. 
When $d+1<2k$ holds, we use $J(d+1,\bar{k})$ instead of $J(d+1,k)$, where $\bar{k}=d+1-k$. 
We would like to find the largest subset $\mathcal{Y}$ of $J(d+1,k)$ with $\mathfrak{l}(\mathbf{x},\mathbf{y}) \in \{2,4\}$ considering Lemma~\ref{lem:l}. We use the notation $Y$ as the subset of $T_d(k,\beta)$ corresponding to $\mathcal{Y} \subset J(d+1,k)$.

For $(d,k)=(7,6)$, we consider a subset $\mathcal{Y}$ of $J(8,2)$ with $\mathfrak{l}(\mathbf{x},\mathbf{y}) \in \{2,4\}$. 
Since $\mathfrak{l}(\mathbf{x},\mathbf{y})=2$ for any $\mathbf{x},\mathbf{y} \in \mathcal{Y}$ with $\mathbf{x} \ne \mathbf{y}$, it follows that $|\mathcal{Y}| \leq 4$. 
This implies that $|R_d \cup Y| \leq 12$, which is smaller than $2d+2$. 
Therefore, for $(d,k)=(7,6)$, there does not exist a 2-distance set $X$ such that $X$ contains $R_d$ and $|X| \geq 2d+2$. 

For $(d,k)=(8,8)$, we consider a subset $\mathcal{Y}$ of $J(9,1)$ with $\mathfrak{l}(\mathbf{x},\mathbf{y}) \in \{2,4\}$. 
Since $\mathfrak{l}(\mathbf{x},\mathbf{y})\in \{ 2,4\}$ is impossible,  we cannot add 2 vectors in $T_d(k,\beta)$ to $R_d$. 
This implies that $|R_d \cup Y| \leq 10$, which is smaller than $2d+2$. 
  Therefore, for $(d,k)=(8,8)$, there does not exist a 2-distance set $X$ such that $X$ contains $R_d$ and $|X| \geq 2d+2$. 

For $(d,k)=(8,5)$, we consider a subset of $J(9,4)$ with $\mathfrak{l}(\mathbf{x},\mathbf{y})\in \{2,4\}$. 
Let $\mathcal{Y}$ be a largest subset of $J(9,4)$ with $\mathfrak{l}(\mathbf{x},\mathbf{y}) \in \{2,4\}$.  
The following lemma is used to determine $\mathcal{Y}$. 
\begin{lemma} \label{lem:op}
Let $J(n,k)$ be the Johnson scheme with distance function $\mathfrak{l}(\mathbf{x},\mathbf{y})$ for $n\geq 2k$.  
Let $X$ be a subset of $J(n,k)$. Suppose $X$ has only distance $k-t$ for some $t$. 
If $nt \geq k^2$ holds, then $|X| \leq n-1$. 
\end{lemma}
\begin{proof}
Let $H$ be the hyperplane which is perpendicular to the all-ones vector and contains the origin.  
Let $f$ be the projection map from $J(n,k)$ to $H$ defined by
\[
f: (\underbrace{1,\ldots,1}_{k},\underbrace{0,\ldots,0}_{n-k} ) \mapsto 
\frac{1}{n} (\underbrace{n-k,\ldots,n-k}_{k},\underbrace{-k,\ldots,-k}_{n-k} ). 
\]
Since $\mathfrak{l}(\mathbf{x},\mathbf{y})=k-t$ for $\mathbf{x},\mathbf{y} \in X$ with $\mathbf{x}\ne \mathbf{y}$, we have $(f(\mathbf{x}),f(\mathbf{y}))=t-k^2/n$, where $(,)$ is 
the usual inner product of $\mathbb{R}^n$. 
Note that $f(X)$ is a $1$-distance set in $S^{n-1}\subset H$.  
If  $t-k^2/n\geq 0$ holds, then the central angle of any two vectors in $f(X)$ is at most $\pi/2$. Therefore $f(X)$ is not the $(n-1)$-dimensional regular simplex, which has central angle $\arccos(-1/(n-1)) >\pi/2$. This implies $|X| =|f(X)|\leq n-1$.  
\end{proof}

If $\mathcal{Y}$ has only distance $\mathfrak{l}(\mathbf{x},\mathbf{y})=2$, then
$|\mathcal{Y}|\leq 8$ by Lemma \ref{lem:op}, and hence $|R_d\cup Y| \leq 17 <2d+2$. 
Therefore, there are $\mathbf{x},\mathbf{y} \in \mathcal{Y}$ such that $\mathfrak{l}(\mathbf{x},\mathbf{y})=4$. 
Without loss of generality, let $\mathbf{x}=(1,1,1,1,0,0,0,0,0), \mathbf{y}=(0,0,0,0,1,1,1,1,0) \in J(9,4)$.  
If  the last entry of a vector $\mathbf{z} \in J(9,4)$ is 1, then 
$\mathfrak{l}(\mathbf{x},\mathbf{z})$ and $\mathfrak{l}(\mathbf{y},\mathbf{z})$ are not in $\{2,4\}$ simultaneously.  
This implies that $\mathcal{Y}$ can be identified with a subset of $J(8,4)$.
Since there does not exist $\mathbf{z} \in \mathcal{Y} \subset J(8,4)$ with $\mathbf{z} \ne \mathbf{x},\mathbf{y}$ such that $\mathfrak{l}(\mathbf{x},\mathbf{z})=4$ or $\mathfrak{l}(\mathbf{y},\mathbf{z})=4$, the 
graph $G=(\mathcal{Y},E)$ with $E=\{(\mathbf{x},\mathbf{y}) \colon\, \mathfrak{l}(\mathbf{x},\mathbf{y})=4\}$ is bipartite. 
A partite set of  $G$ is a subset $Z$ of $\mathcal{Y}$ that has only distance $\mathfrak{l}(\mathbf{x},\mathbf{y})=2$. 
Since $\mathcal{Y}$ can be identified with a subset of $J(8,4)$, we have $|Z| \leq 7$ by Lemma~\ref{lem:op}.   
Therefore $\mathcal{Y}=Z\cup \overline{Z}$ is largest, where 
\begin{multline*}
Z=\{(1,1,1,1,0,0,0,0,0), (1,1,0,0,1,1,0,0,0),(1,0,1,0,1,0,1,0,0),\\
(1,0,0,1,1,0,0,1,0),(1,1,0,0,0,0,1,1,0,0),
(1,0,1,0,0,1,0,1,0),(1,0,0,1,0,1,1,0,0)\}
\end{multline*}
and
\begin{multline*}
\overline{Z}=\{(0,0,0,0,1,1,1,1,0), (0,0,1,1,0,0,1,1,0),(0,1,0,1,0,1,0,1,0),\\
(0,1,1,0,0,1,1,0,0),(0,0,1,1,1,1,0,0,0),
(0,1,0,1,1,0,1,0,0),(0,1,1,0,1,0,0,1,0)\}. 
\end{multline*} 
These sets are obtained from the Hadamard matrix of order 8, and the set $\mathcal{Y}$ is unique up to permutations of coordinates. 
The set $\mathcal{Y}$ can be identified with a subset of $ T_8(5,-1/2)$. Each point in $Y$ has the last coordinate $1/3$. Let $\mathbf{z} \in T_8(8,-1/2)$ be the point that has the last coordinate $-1/3$, namely
\[
\mathbf{z}=\frac{1}{6}(1,1,1,1,1,1,1,1,-2). 
\]
The set $R_8 \cup Y \cup \{\mathbf{z}\}$ is a maximal $2$-distance set with 24 points.  
This 2-distance set is largest for $\alpha <2$ since at most 1 point in $T_8(8,-1/2)$ can be added to $R_8$ as proved above.

\subsection{LRS ratio $s=2$, distances $\sqrt{2}, 2$ ($\alpha=4$)}
By Theorem~\ref{lem:key},  
$
d=k+2+4/(k-1)
$. 
Since $d$, $k$ are natural numbers, the possible $(d,k)$ are $(7,3)$, $(8,2)$, or $(8,5)$. 
Here $k=2$ and $k'=5$ for $d=8$, and $k=k'=3$ for $d=7$. We would like to find the largest subset $\mathcal{Y}$ of $J(d+1,k)$ with $\mathfrak{l}(\mathbf{x},\mathbf{y}) \in \{1,2\}$ considering Lemma~\ref{lem:l}.

For $(d,k)=(7,3)$, we consider a subset $\mathcal{Y}$ of $J(8,3)$ with $\mathfrak{l}(\mathbf{x},\mathbf{y}) \in \{1,2\}$. 
The subset $\mathcal{Y}$ is a 1-intersecting family and the size of $\mathcal{Y}$ is at most 21 \cite{EKR61,W84}. Here $\mathcal{Y}$ is a {\it $t$-intersecting family} if $|\mathbf{x}\cap \mathbf{y}| \geq t$ for any $\mathbf{x},\mathbf{y} \in \mathcal{Y}$.   
The unique largest set is $\mathcal{Y}=\{(x_1,\ldots,x_8) \in J(8,3) \mid x_1=1\}$, up to permutations of coordinates \cite{EKR61,W84}. 
The set $R_7 \cup Y$ is a maximal 2-distance set with $29$ points. Indeed, $R_7 \cup Y$
is largest in all $2$-distance sets in $\mathbb{R}^7$  \cite{ES66}. 

For $(d,k)=(8,2)$, we consider a largest subset $\mathcal{Y}$ of $J(9,2)$ with $\mathfrak{l}(\mathbf{x},\mathbf{y}) \in \{1,2\}$. 
In this case, $\mathcal{Y}=J(9,2)$. 
The set $R_8 \cup J(9,2)$, whose size is 45,  is largest in all $2$-distance sets in $\mathbb{R}^8$ \cite{L97}.   

For $(d,k)=(8,5)$, we consider a subset $\mathcal{Y}$ of $J(9,4)$ with $\mathfrak{l}(\mathbf{x},\mathbf{y}) \in \{1,2\}$. 
Since $\mathcal{Y}$ is a $2$-intersecting family, 
the unique largest set is $\mathcal{Y}=\{(x_1,\ldots,x_9) \in J(9,4) \mid x_1=x_2=1 \}$ \cite{AK97}. 
The set $Y$ can be identified with a subset of $T_8(5,1)$. 
 Each point $\mathbf{z} \in T_8(2,1)$ cannot be added to $R_8\cup Y$ while maintaining the $2$-distance condition.  
Therefore $R_8 \cup Y$ is a maximal $2$-distance set with  $30$ points. 

The largest 2-distance set that contains $R_d$ with at least $2d+2$ points are summarized as follows. 
\begin{theorem}
Let $X$ be the largest 2-distance set in $\mathbb{R}^d$ that contains $R_d$ with distances $\sqrt{2}$ and $\sqrt{\alpha}$, where $\sqrt{2}$ is the distance of $R_d$.  
 If the LRS ratio is $2$, then $d=7,8$, $\alpha=1,4$ and we have the following. 
\begin{enumerate}
\item $|X| =12$ for $d=7$ and $\alpha=1$. 
\item $|X| =24$ for $d=8$ and $\alpha=1$.
\item $|X| =29$ for $d=7$ and $\alpha=4$. 
\item $|X| =45$ for $d=8$ and $\alpha=4$. 
\end{enumerate}
\end{theorem}

\subsection{Adding a quasi-symmetric design}
If the LRS ratio is greater than 2, it is hard to determine 
the largest 2-distance subset of $T_d(k,\beta)$ that can be added to $R_d$ because $d$ becomes large and $T_d(k,\beta)$ has large size. We try to add the representation of a quasi-symmetric design to $R_d$. A {\it quasi-symmetric design} is a combinatorial design whose blocks have intersections of only two sizes \cite{N82, SSb}.  
The block set of a quasi-symmetric design is identified with a subset of the Johnson scheme. Namely we add the characteristic vectors of the blocks to $R_d$.  
The block graph of a quasi-symmetric design is a strongly regular graph. The characteristic vectors of the blocks are the Euclidean representation of the corresponding strongly regular graph. 
The LRS ratio $s$ of the representation is equal to the absolute value of the smallest eigenvalue of the strongly regular graph \cite{BB05}. 
For $s=3$, a quasi-symmetric design whose block graph has smallest eigenvalue $-3$ is investigated in \cite{NPS17}.   
After modifying the representation of a quasi-symmetric design with eigenvalue $-s$,  
we can add the set to $R_d$.  
We check the maximality of the new 2-distance set 
by trying to add other vectors in $T_d(k,\beta)$ or $T_d(k',\beta)$, where $k'$ is defined in Lemma~\ref{lem:m}.

\subsubsection{Strongly resolvable design}
A {\it strongly resolvable design} is a quasi-symmetric design
whose block graph is the complete multipartite regular graph. 
For prime power $s$, a strongly resolvable 2-$(s^3,s^2,s+1)$ design 
is obtained from the affine space $AG(3,s)$.  
The point set of $AG(3,s)$ is $\mathbb{F}_s^3$, where 
$\mathbb{F}_s$ is the finite field of order $s$. The block set $\mathfrak{B}$ of $AG(3,s)$ is the set of all $2$-dimensional affine subspaces of $\mathbb{F}_s^3$, and its size is $s^3+s^2+s$. 
Let $\mathbf{v}_B$ be the characteristic vector of $B \in \mathfrak{B}$. The length of $\mathbf{v}_B$ is $s^3$. 
The distance $\mathfrak{l}(\mathbf{v}_{B_1}, \mathbf{v}_{B_2})$ is $s^2$ or $s(s-1)$ for  any $B_1,B_2 \in B$ with $B_1 \ne B_2$, 
and the LRS ratio is $s$. 

By Theorem~\ref{lem:key}, for $\alpha <2$ and $k=s^2+s-1$, 
we have $\beta=-1/s$ and $d+1=(s-1)(s+1)^2$. 
By Lemma~\ref{lem:l}, if distinct vectors $\mathbf{x}, \mathbf{y} \in J(d+1,k)$ 
can be added to $R_d$, then $\mathfrak{l}(\mathbf{x},\mathbf{y}) \in \{s^2,s(s-1)\}$, where $J(d+1,k)$ is identified with $T_d(k,\beta)$. For the block set $\mathfrak{B}$ of $AG(3,s)$, we define
\begin{equation} \label{eq:B'}
\mathfrak{B}'=\{(\underbrace{1,\ldots, 1}_{s-1},\underbrace{0,\ldots,0}_{s^2-2s}, \mathbf{v}_B)
\colon\, B \in \mathfrak{B}\} \subset J(d+1,k).
\end{equation}
By Lemma~\ref{lem:l}, $R_d\cup \mathfrak{B}'$ is a $2$-distance set. 
For $d+1=(s-1)(s+1)^2$, we have another choice of $k$, namely $k'=s^3$ by Lemma~\ref{lem:two_k}. 
The vector $\mathbf{x}_0 \in T_d(k',\beta)$ is defined to be the vector  corresponding to 
\begin{equation} \label{eq:3.6}
(\underbrace{0,\ldots,0}_{s^2-s-1}, \underbrace{1,\ldots,1}_{s^3}) \in J(d+1,k'). 
\end{equation}
By Lemma~\ref{lem:m}, the set $R_d\cup \mathfrak{B}'\cup \{\mathbf{x}_0\}$ is a $2$-distance set with 
$2s^2(s+1)$ points, and $|R_d\cup \mathfrak{B}'\cup\{\mathbf{x}_0\}|\geq 2d+2$. 
 
We use the following lemma in order to show the maximality of $R_d\cup \mathfrak{B}' \cup\{\mathbf{x}_0\}$. 
\begin{lemma}\label{lem:resol}
Let $s$ be an integer at least $2$. 
Let $a$, $b$ be non-negative integers such that $a-b=s$ and $b<s$.  
Let $(P,\mathfrak{B})$ be a strongly resolvable 2-$(s^3,s^2,s+1)$ design, which is a quasi-symmetric design such that  $|B_1\cap B_2|=\{0, s\}$ for any $B_1,B_2 \in \mathfrak{B}$ with $B_1 \ne B_2$.  
Then there does not exist a non-empty subset $S$ of $P$ such that 
$|S \cap B| \in \{a, b\}$ for each $B \in \mathfrak{B}$.      
\end{lemma}
\begin{proof}
Assume $S$ is a subset of $P$ such that 
$|S \cap B| \in \{a, b\}$ for each $B \in \mathfrak{B}$.    
Since the 2-$(s^3,s^2,s+1)$ design is strongly resolvable,  
there exist $s$ blocks $B_1, \ldots, B_s$ such that 
$P=\bigsqcup_{i=1}^s B_i$, that is disjoint union. 
Let $S_i= S \cap B_i$.  Suppose 
$|S_i|=a$ for $1\leq i \leq t$, and 
 $|S_i|=b$ for $t+1\leq i \leq s$. 
For fixed $x_0 \in S_1$, there exist $s(s+1)$ blocks that contain $x_0$ except for $B_1$. Let $\mathfrak{C}=\{Z_1,\ldots, Z_m \}$ be the set of the $s(s+1)$ blocks, where $m=s(s+1)$. 
Suppose $|S\cap Z_i|=a$ for $1\leq i \leq u$, and 
 $|S\cap Z_i|=b$ for $u+1\leq i \leq m$. 
For  each $y \in S_1\setminus \{x_0\}$, 
there exist $s$ blocks in $\mathfrak{C}$ that contain $y$. 
For each $z \in S_i$ with $i \ne 1$, there exist
$s+1$ blocks in $\mathfrak{C}$ that contain $z$. 
Counting the points in $S$ including duplication, 
\[
\sum_{i=1}^m |S \cap Z_i|= \sum_{i=1}^s \sum_{x \in S_i}
 |\{j\colon\, x \in Z_j \}|. 
\] 
From this equality, for $t \ne 0$ 
\begin{equation} \label{eq:au1}
au+b(m-u)=s(s+1)+(a-1)s+a(s+1)(t-1)+b(s+1)(s-t), 
\end{equation}
and for $t =0$
\begin{equation}\label{eq:au2}
au+b(m-u)=s(s+1)+(b-1)s+b(s+1)(s-1). 
\end{equation}

Equation \eqref{eq:au1} implies that 
$
a \equiv 0 \pmod{s} 
$
and
$
a \equiv u+1 \pmod{s+1}  
$.
From $a-b=s$ and $
a \equiv 0 \pmod{s} 
$, we have $b \equiv 0 \pmod{s}$. 
From $b<s$ and $b \equiv 0 \pmod{s}$,  we have $a=s$ and $b=0$.      
For each $i$, it follows $|S \cap Z_i| \geq 1$ since $x_0$ is in $S \cap Z_i$.  The number of $i$ such that 
$|S \cap Z_i|=b$ is $m-u$. Thus $b=0$ implies $u=m$, which contradicts $0\equiv s(s+1)=m= u \equiv a-1 \equiv s-1 \pmod{ s+1}$. 

Equation \eqref{eq:au2} implies that 
$
b \equiv 0 \pmod{s} 
$
and
$
b \equiv u+1 \pmod{s+1}  
$.
From $b<s$ and $b \equiv 0 \pmod{s}$,  we have $a=s$ and $b=0$. 
 This implies $u=m$, which contradicts $0\equiv s(s+1)=m= u \equiv b-1 \equiv -1 \pmod{ s+1}$. 

The lemma therefore follows.  
\end{proof}

\begin{theorem}
Let $s$ be a prime power, $d=(s-1)(s+1)^2-1$, $ \mathbf{x}_0$ defined in \eqref{eq:3.6}, and $\mathfrak{B}'$ defined in \eqref{eq:B'}. 
Then the $2$-distance set $R_d\cup \mathfrak{B}' \cup \{\mathbf{x}_0\}$ is maximal. 
\end{theorem}
\begin{proof}
Suppose $\mathbf{x} \in T_d(k,\beta)$ can be added to $R_d\cup \mathfrak{B}' \cup \{\mathbf{x}_0\}$ while maintaining the $2$-distance condition, where $\beta=-1/s$. By Lemma~\ref{lem:two_k}, it follows that $\mathbf{x}$ is in $T_d(k,\beta)$ or $T_d(k',\beta)$ for $k=s^2+s-1$ and $k'=s^3$. 
Assume $\mathbf{x} \in T_d(k',\beta)$. 
Since the number of entries $0$ of $\mathbf{x}_0$ in \eqref{eq:3.6} is $s^2-s-1$, we have $\mathfrak{l}(\mathbf{x}_0,\mathbf{x}) \leq s^2-s-1$. Thus $\mathfrak{l}(\mathbf{x}_0,\mathbf{x}) \not\in \{s^2,s(s-1)\}$ and  hence
$R_d \cup \{\mathbf{x}_0, \mathbf{x}\}$ is not 2-distance by Lemma~\ref{lem:l}.  
Thus $\mathbf{x} \in T_d(k,\beta)$. 
 Now $\mathbf{x}$ is identified with a point of $J(d+1,k)$, and so is a point $\mathbf{y}$ of $\mathfrak{B}'$. Then it follows that $|\mathbf{x} \cap \mathbf{y}| \in \{s-1,2s-1\}$ for each $\mathbf{y} \in \mathfrak{B}'$ 
because $\mathfrak{l}(\mathbf{x}, \mathbf{y}) \in \{s^2,s(s-1)\}$ by Lemma~\ref{lem:l}. 
Let $P$ be the point set of $AG(3,s)$, and $\mathfrak{B}$ the block set of $AG(3,s)$. Let $X=\mathbf{x} \cap P$.  
For some integers $a,b$ with $a-b=s$ and $b<s$, it follows that $|X \cap B|\in \{a,b\}$ for each $B \in \mathfrak{B}$. 
By Lemma~\ref{lem:resol}, there does not exist such $X$.
This implies the assertion. 
\end{proof}
\label{sec:3.3.1}

\subsubsection{$4$-$(23,7,1)$ Witt design or its complement}
Let $\alpha=3$, namely $s=3$. 
The set $T_d(k,\beta)$ is identified with $J(d+1,k)$. 
Let  
$
\mathfrak{m}(\mathbf{x},\mathbf{y})=|\{i \mid x_i=y_i=1\}|
$ for $\mathbf{x}=(x_1,\ldots, x_{d+1}), \mathbf{y}=(y_1,\ldots, y_{d+1}) \in J(d+1,k)$.  
If $\mathbf{x},\mathbf{y} \in J(d+1,k)$ can be added to $R_d$, 
then $\mathfrak{m}(\mathbf{x},\mathbf{y}) \in \{ k-6,k-4\}$ by Lemma~\ref{lem:l}.  
By Theorem~\ref{lem:key} (2), 
the dimensions are only $(d,k)=(23,10),
(24,8),(24,13),(26,7),(26,16),(31,6),(31,22),(48,5),(48,40)$. 
By Lemma~\ref{lem:m} (1), if $\mathbf{x} \in J(d+1,k)$ and
$\mathbf{y} \in J(d+1,k')$ can be added to $R_d$, then 
$\mathfrak{m}(\mathbf{x},\mathbf{y}) \in \{6,9\}$.

Let $\mathfrak{B}$ be the block set of the $4$-$(23,7,1)$ Witt design \cite{GS70}. The size of $\mathfrak{B}$ is 253.  Let $\bar{\mathfrak{B}}$ be the block set of the complement of the $4$-$(23,7,1)$ design. 
By exhaustive computer search, we can check 
$|\{  \mathfrak{m}(\mathbf{x},\mathbf{y}) \colon\, \mathbf{x} \in \mathfrak{B}\}| \geq 3$ for each $\mathbf{y} \in J(23,k)$ with 
$2 \leq k \leq 21$. 
Note that if $|\{  \mathfrak{m}(\mathbf{x},\mathbf{y}) \colon\, \mathbf{x} \in \mathfrak{B}\}| \geq 3$ holds, then
$|\{  \mathfrak{m}(\mathbf{x},\mathbf{y}) \colon\, \mathbf{x} \in \bar{\mathfrak{B}}\}| \geq 3$ also holds. 
They are used in the following (i)--(v). 


\noindent
(i) $d=24$, $k=8$, $\alpha=3$, $k'=13$

We define
\[
\mathfrak{B}'=\{(1,0, \mathbf{v}_B)
\colon\, B \in \mathfrak{B}\} \subset J(25,8). 
\]  
Let $\mathbf{z}'=(*, *, \mathbf{z}) \in J(25,8)$, where 
$\mathbf{z} \in J(23,8-i)$ for some $i\in \{0,1,2\}$. 
If 
$|\{ \mathfrak{m}(\mathbf{z}',\mathbf{x}') \colon\, \mathbf{x}' \in \mathfrak{B}'\}|\leq 2 $ holds, then  
$|\{\mathfrak{m}(\mathbf{z},\mathbf{v}_B) \colon\, B \in \mathfrak{B}\}| \leq 2$, but it is impossible.  
For $k'=13$, let $\mathbf{z}'=(*, *, \mathbf{z}) \in J(25,13)$, 
where $\mathbf{z} \in J(23,13-i)$ for some $i\in \{0,1,2\}$. 
If 
$|\{ \mathfrak{m}(\mathbf{z}',\mathbf{x}') \colon\, \mathbf{x}' \in \mathfrak{B}'\}|\leq 2 $ holds, then  
$|\{\mathfrak{m}(\mathbf{z},\mathbf{v}_B) \colon\, B \in \mathfrak{B}\}| \leq 2$, but it is impossible.  
These imply that $X=R_{24} \cup \mathfrak{B}'$ is maximal as a 2-distance set in $\mathbb{R}^{24}$. The set $X$ has 278 points.  

\noindent 
(ii) $d=26$, $k=7$, $\alpha=3$, $k'=16$

We define
\[
\mathfrak{B}'=\{(0,0,0,0, \mathbf{v}_B)
\colon\, B \in \mathfrak{B}\} \subset J(27,7).
\]
Let $\mathbf{z}'=(*, *, *,*,\mathbf{z}) \in J(27,7)$, where 
$\mathbf{z} \in J(23,7-i)$ for some $i\in \{0,1,\ldots ,4\}$. 
If 
$|\{ \mathfrak{m}(\mathbf{z}',\mathbf{x}') \colon\, \mathbf{x}' \in \mathfrak{B}'\}|\leq 2 $ holds, then  
$|\{\mathfrak{m}(\mathbf{z},\mathbf{v}_B) \colon\, B \in \mathfrak{B}\}| \leq 2$, but it is impossible.  
For $k'=16$, let $\mathbf{z}'=(*,*,*, *, \mathbf{z}) \in J(27,16)$, 
where $\mathbf{z} \in J(23,16-i)$ for some $i\in \{0,1,\ldots ,4\}$. 
If 
$|\{ \mathfrak{m}(\mathbf{z}',\mathbf{x}') \colon\, \mathbf{x}' \in \mathfrak{B}'\}|\leq 2 $ holds, then  
$|\{\mathfrak{m}(\mathbf{z},\mathbf{v}_B) \colon\, B \in \mathfrak{B}\}| \leq 2$, but it is impossible.  
These imply that $X=R_{26} \cup \mathfrak{B}'$ is maximal as a 2-distance set in $\mathbb{R}^{24}$. The set $X$ has 280 points.  

\noindent
(iii) $d=26$, $k=16$, $\alpha=3$, $k'=7$

We define
\[
\bar{B}'=\{(0,0,0,0, \mathbf{v}_B)
\colon\, B \in \bar{\mathfrak{B}}\} \subset J(27,16).
\]
Let $\mathbf{z}'=(*, *,*,*, \mathbf{z}) \in J(27,16)$, where 
$\mathbf{z} \in J(23,16-i)$ for some $i\in \{0,1,\ldots ,4\}$. 
If 
$|\{ \mathfrak{m}(\mathbf{z}',\mathbf{x}') \colon\, \mathbf{x}' \in \bar{\mathfrak{B}}'\}|\leq 2 $ holds, then  
$|\{\mathfrak{m}(\mathbf{z},\mathbf{v}_B) \colon\, B \in \bar{\mathfrak{B}}\}| \leq 2$, but it is impossible.  
For $k'=7$, let $\mathbf{z}'=(*, *,*,*, \mathbf{z}) \in J(27,7)$, 
where $\mathbf{z} \in J(23,7-i)$ for some $i\in \{0,1,\ldots,4\}$. 
If 
$|\{ \mathfrak{m}(\mathbf{z}',\mathbf{x}') \colon\, \mathbf{x}' \in \bar{\mathfrak{B}}'\}|\leq 2 $ holds, then  
$|\{\mathfrak{m}(\mathbf{z},\mathbf{v}_B) \colon\, B \in \bar{\mathfrak{B}}\}| \leq 2$, but it is impossible.  
These imply that $X=R_{26} \cup \bar{\mathfrak{B}}'$ is maximal as a 2-distance set in $\mathbb{R}^{26}$. The set $X$ has 280 points. 
The $2$-distance set $X$ is isometric to the set in (ii).  

\noindent
(iv) $d=31$, $k=22$, $\alpha=3$, $k'=6$

We 
\[
\bar{\mathfrak{B}}'=\{(1,1,1,1,1,1,0,0,0, \mathbf{v}_B)
\colon\, B \in \bar{\mathfrak{B}}\} \subset J(32,22).
\]
Let $\mathbf{z}'=(\underbrace{*, \ldots *}_{9}, \mathbf{z}) \in J(32,22)$, where 
$\mathbf{z} \in J(23,22-i)$ for some $i\in \{0,1,\ldots ,9\}$. 
If 
$|\{ \mathfrak{m}(\mathbf{z}',\mathbf{x}') \colon\, \mathbf{x}' \in \bar{\mathfrak{B}}'\}|\leq 2 $ holds, then  
$|\{\mathfrak{m}(\mathbf{z},\mathbf{v}_B) \colon\, B \in \bar{\mathfrak{B}}\}| \leq 2$, but it is impossible.  

For $k'=6$, let $\mathbf{z}'=(\underbrace{*, \ldots *}_{9}, \mathbf{z}) \in J(32,6)$, 
where $\mathbf{z} \in J(23,6-i)$ for some $i\in \{0,1,\ldots ,6\}$. 
For $\mathbf{z} \in J(23,6-i)$ with $i\in \{0,1,\ldots,4\}$,  if 
$|\{ \mathfrak{m}(\mathbf{z}',\mathbf{x}') \colon\, \mathbf{x}' \in \bar{\mathfrak{B}}'\}|\leq 2 $ holds, then  
$|\{\mathfrak{m}(\mathbf{z},\mathbf{v}_B) \colon\, B \in \bar{\mathfrak{B}}\}| \leq 2$, but it is impossible.  
It should hold that $\mathfrak{m}(\mathbf{z}',\mathbf{x}') \in \{6,9\}$ for each $\mathbf{x}' \in \bar{\mathfrak{B}}'$, and $\mathfrak{m}(\mathbf{z},\mathbf{x}) \in \{6-j,9-j\}$ for each $\mathbf{x} \in \bar{\mathfrak{B}}$ for some $j \in \{0,1,\ldots 6\}$. 
 For $\mathbf{z} \in J(23,1)$, $\mathfrak{m}(\mathbf{z},\mathbf{x}) \not\in \{6-j,9-j\}$ for each $\mathbf{x} \in \bar{\mathfrak{B}}$ for any $j \in \{0,1,\ldots 6\}$. 
The vector 
\[
\mathbf{x}_0=(1,1,1,1,1,1,\underbrace{0,\ldots, 0}_{26}), 
\]
can be added to $R_{31} \cup \bar{\mathfrak{B}}'$
while maintaining the 2-distance condition.

These imply that $X=R_{31} \cup \bar{\mathfrak{B}}' \cup \{\mathbf{x}_0\}$ is maximal as a 2-distance set in $\mathbb{R}^{31}$. The set $X$ has 286 points.  

\noindent
(v) $d=48$, $k=40$, $\alpha=3$, $k'=5$

We define
\[
\bar{\mathfrak{B}}'=\{(\underbrace{1,\ldots,1}_{24},0,0, \mathbf{v}_B)
\colon\, B \in \bar{\mathfrak{B}}\} \subset J(49,40).
\]

Let $\mathbf{z}'=(\underbrace{*, \ldots, *}_{26}, \mathbf{z}) \in J(49,40)$, where 
$\mathbf{z} \in J(23,23-i)$ for some $i\in \{0,1,\ldots ,9\}$. 
For $\mathbf{z} \in J(23,23-i)$ with $i\in \{2, \ldots ,9\}$ if 
$|\{ \mathfrak{m}(\mathbf{z}',\mathbf{x}') \colon\, \mathbf{x}' \in \bar{\mathfrak{B}}'\}|\leq 2 $ holds, then  
$|\{\mathfrak{m}(\mathbf{z},\mathbf{v}_B) \colon\, B \in \bar{\mathfrak{B}}\}| \leq 2$, but it is impossible.  
For any 
$\mathbf{x},\mathbf{y} \in \bar{\mathfrak{B}}'$, we have $\mathfrak{m}(\mathbf{x},\mathbf{y}) \in \{ 34, 36\}$. 
If $\mathbf{z}'$ can be added to $R_{48} \cup \bar{\mathfrak{B}}'$
while maintaining the 2-distance condition, then $\mathfrak{m}(\mathbf{z}',\mathbf{x}') \in \{34,36\}$ for each $\mathbf{x}' \in \bar{\mathfrak{B}}'$. It is impossible for $\mathbf{z} \in J(23,23)$ or $J(23,22)$. 

For $k'=5$ and $\mathbf{z}' \in J(49,5)$, it should hold that  $\mathfrak{m}(\mathbf{z}',\mathbf{x}') \in \{6,9\}$ for each $\mathbf{x}' \in \bar{\mathfrak{B}}'$, but it is impossible. 

These imply that $X=R_{48} \cup \bar{\mathfrak{B}}'$ is maximal as a 2-distance set in $\mathbb{R}^{48}$. The set $X$ has 302 points.

\subsubsection{$3$-$(22,6,1)$ Witt design, $d=31$, $k=6$, $\alpha=3$, $k'=22$
}

Let $\mathfrak{B}$ be the block set of the complement of the $3$-$(22,6,1)$ Witt design \cite{GS70}. The size of $\mathfrak{B}$ is 77. We define
\[
\mathfrak{\mathfrak{B}}'=\{(\underbrace{0,\ldots,0}_{10}, \mathbf{v}_B)
\colon\, B \in \mathfrak{B}\} \subset J(32,6).
\]
Note that $\mathfrak{m}(\mathbf{x},\mathbf{y}) \in \{ 0, 2\}$ for any 
$\mathbf{x},\mathbf{y} \in \mathfrak{B}'$. 
By exhausted computer search, 
we can show that for each $i\in \{0,1, \ldots, 5\}$, there does not exist $\mathbf{x} \in J(22,6-i)$  such that $\mathfrak{m}(\mathbf{x},\mathbf{y}) \in \{0,2\}$ for each $\mathbf{y} \in \mathfrak{B}$.  
We can also show that for each $i\in \{1,\ldots, 10\}$, there does not exist $\mathbf{x} \in J(22,22-i)$ such that $\mathfrak{m}(\mathbf{x},\mathbf{y}) \in \{6,9\}$ for each $\mathbf{y} \in \mathfrak{B}$.  The vector 
\[
\mathbf{x}_0=(\underbrace{0,\ldots, 0}_{10},\underbrace{1,\ldots, 1}_{22})
\]
can be added to $R_{31} \cup \mathfrak{B}'$ while maintaining the 2-distance condition. 
These imply that $X=R_{31} \cup \mathfrak{B}' \cup \{\mathbf{x}_0\}$ is maximal as a 2-distance set in $\mathbb{R}^{31}$. The set $X$ has 110 points.  

\subsubsection{$2$-$(21,7,12)$ design, $d=23$, $k=10$, $\alpha=3$, $k'=10$}


Let $\mathfrak{B}$ be the block set of the $2$-$(21,7,12)$ design \cite{GS70,T86}. The size of $\mathfrak{B}$ is 120. 
 We define
\[
\mathfrak{B}'=\{(1,1,1, \mathbf{v}_B)
\colon\, B \in \mathfrak{B}\} \subset J(24,10).
\]
Note that $\mathfrak{m}(\mathbf{x},\mathbf{y}) \in \{ 4, 6\}$ for any 
$\mathbf{x},\mathbf{y} \in \mathfrak{B}'$.  
By exhausted computer search, 
we can show that for each $i\in \{0,1,2,3\}$ and each $j \in \{0,1,2,3\}$, there does not exist $\mathbf{x} \in J(21,10-i)$  such that $\mathfrak{m}(\mathbf{x},\mathbf{y}) \in \{4-j,6-j\}$ for each $\mathbf{y} \in \mathfrak{B}$.  
These imply that $X=R_{23} \cup \mathfrak{B}'$ is maximal as a 2-distance set in $\mathbb{R}^{23}$. The set $X$ has 144 points.  




\section{Concluding remarks}
We are interested in determining the largest possible $m$-distance set in $\mathbb{R}^d$ for given $m$ and $d$. In this paper, 
we assumed that a $2$-distance set $X$ in $\mathbb{R}^d$ contains a $d$-dimensional regular simplex and has size at least $2d+2$, and we attempted to determine the largest $X$. 
For the LRS ratio $2$, the largest 2-distance sets $X$ exist only for $d=7,8$, and we classified the sets. For the LRS ratio $s\geq 3$, the situation is comparatively more complicated. It is difficult to determine the largest subset $Y$ of $T_d(k,\beta) \cup T_d(k',\beta)$ that satisfies the conditions in Lemmas~\ref{lem:l} and \ref{lem:m}. Determining the largest sets $Y$ is equivalent to determining the largest 2-distance sets $X$ that contain a regular simplex. The main reason for the difficulty is that we have no means of systematic construction of non-spherical large $m$-distance sets in general. 
For $s=2$, the possible dimensions are relatively small and we have large $2$-distance sets. This is very helpful in determining the largest 2-distance sets $X$ that contain a regular simplex. 
However, for $s\geq 3$, the possible dimensions are so large that it is almost impossible to determine the largest sets $Y$ so far. 
In this study, we found large sets $Y$ from the representations of  quasi-symmetric designs. 
For $s=3$, fortunately, there are several large 2-distance sets $Y$ obtained from the previous research \cite{NPS17}.  
We can prove the maximality of these sets after adding a suitable point, if needed. 
For any prime power $s$, we constructed a maximal $2$-distance set containing a regular simplex by adding the representation of a strongly resolvable design. 
However the largest 2-distance set $X$ that contains a regular simplex for $s\geq 3$ has still not been determined. 

Based on the aforementioned situation, we would like to suggest further problems as follows: 
\begin{problem}
 Determine the largest possible 2-distance sets $X$ in $\mathbb{R}^d$ that contain a $d$-dimensional regular simplex for a given $d$. 
\end{problem}
\begin{problem}
 Determine the largest possible 2-distance sets $X$ in $\mathbb{R}^d$ that contain a $d$-dimensional regular simplex and  have size at least $2d+2$ for given $d$ and LRS ratio $s\geq 3$. 
This problem is equivalent to finding the largest subset of $T_d(k,\beta) \cup T_d(k',\beta)$ that satisfies the conditions in Lemmas~\ref{lem:l} and \ref{lem:m}.
 \end{problem}
\begin{problem}
Construct maximal $m$-distance sets in $\mathbb{R}^d$ that contain a $d$-dimensional regular simplex. 
In particular, construct the maximal $m$-distance sets whose LRS ratio is an integer. 
\end{problem}

\begin{problem}
Find a strongly regular graph whose Euclidean representation can be added to a regular simplex while maintaining the 2-distance condition. 
\end{problem}

\bigskip

\noindent
\textbf{Acknowledgments.} 
The authors would like to thank anonymous referees for  suggesting the way of the exposition of this paper. 
Nozaki is supported by JSPS KAKENHI Grant Numbers 17K05155,  18K03396, 19K03445, and 20K03527. 
Shinohara is supported by JSPS KAKENHI Grant Number  18K03396.

\end{document}